\documentclass[11pt,reqno]{amsart}
\usepackage{graphicx,amsfonts,amsmath,amsthm,verbatim,amssymb,lineno,titletoc,enumerate,mathrsfs}
\usepackage{tikz}
\usetikzlibrary{decorations.pathreplacing}
\usepackage{bbm}
\usepackage{hyperref}
\hypersetup{colorlinks}
\pdfoutput=1

\newcommand{\floor}[1]{\left\lfloor {#1} \right\rfloor}

\newcommand{\old}[1]{}

\newtheorem{theorem}{Theorem}[section]
\newtheorem{proposition}[theorem]{Proposition}
\newtheorem{lemma}[theorem]{Lemma}
\newtheorem{corollary}[theorem]{Corollary}
\newtheorem{conjecture}[theorem]{Conjecture}

\theoremstyle{remark}
\newtheorem*{remark}{Remark}
\newtheorem*{example}{Example}

\numberwithin{counter}{section}

\theoremstyle{definition}
\newtheorem{definition}[theorem]{Definition}

\def\liminf{\mathop{\rm lim\,inf}\limits}

\newcommand{\R}{\mathbb{R}}
\newcommand{\N}{\mathbb{N}}
\newcommand{\Z}{\mathbb{Z}}
\newcommand{\E}[1]{\mathbb{E}\left[ #1 \right]}
\newcommand{\Esub}[2]{\mathbb{E}_{ #1}\left[ #2 \right]}
\renewcommand{\Pr}[1]{\mathbb{P}\left( #1 \right)} %Probability measure
\newcommand{\Cond}[2]{\mathbb{E}\left[#1\mid #2\right]}
\newcommand{\Prsub}[2]{\mathbb{P}_{ #1 }\left( #2 \right)}
\newcommand{\ind}{\mathbbm{1}} % indicator function
\newcommand{\trans}[2]{\overset{ #1 }{\rightsquigarrow}_{ #2 }}
\newcommand{\IID}{\textrm{i.i.d.\,}}
\newcommand{\PP}{\mathbb{P}}
\newcommand{\EE}{\mathbb{E}}
\newcommand{\eps}{\epsilon}
\newcommand{\Var}{\mathrm{Var}}
\newcommand{\as}{\mathrm{a.s.}}
\newcommand{\surv}{\mathcal{S}} %event of survival
\newcommand{\onesurv}{\mathbbm{1}_{\mathcal{S}}}
\newcommand{\one}{\textbf{1}}

\title{Branching in a Markovian Environment}
\author{Lila Greco and Lionel Levine}
\date{June 21, 2021}

\thanks{LL was partially supported by NSF grant \href{https://www.nsf.gov/awardsearch/showAward?AWD_ID=1455272}{DMS-1455272}.}

\address{Lila Greco, Department of Mathematics, Cornell University, Ithaca, NY 14853.}
\address{Lionel Levine, Department of Mathematics, Cornell University, Ithaca, NY 14853. {\tt \url{https://pi.math.cornell.edu/~levine}}}

\keywords{extinction matrix, martingale central limit theorem, matrix generating function, stochastic abelian network}
\subjclass[2010]{
60J80, %Branching processes (Galton-Watson, birth-and-death, etc.)
60J10,  %Markov chains (discrete-time Markov processes on discrete state spaces)
60K37, %Processes in random environments
60F05, %Central limit and other weak theorems
15A24, %Matrix equations and identities
15B51 %Stochastic matrices
}

\begin{document}

\maketitle

\begin{abstract}
A branching process in a Markovian environment consists of an irreducible Markov chain on a set of ``environments'' together with an offspring distribution for each environment. At each time step the chain transitions to a new random environment, and one individual is replaced by a random number of offspring whose distribution depends on the new environment. We give a first moment condition that determines whether this process survives forever with positive probability. On the event of survival we prove a law of large numbers and a central limit theorem for the population size. We also define a matrix-valued generating function for which the extinction matrix (whose entries are the probability of extinction in state $j$ given that the initial state is $i$) is a fixed point, and we prove that iterates of the generating function starting with the zero matrix converge to the extinction matrix.
\end{abstract}

\section{Introduction}

Consider a population of identical individuals whose reproduction is a function of their environment. The environment follows an irreducible Markov chain on a finite state space. At each time step, the Markov chain transitions to a new random environment, and one individual is replaced by a random number of offspring whose distribution depends on the environment.  We call this a \textbf{branching process in a Markovian environment (BPME)}.

To state our main results, we introduce the quantity
	\[ \mu := \sum_{i \in S} \pi_i \mu_i \]
where $S$ is the set of possible environments, $\vec{\pi}$ is the unique stationary distribution of the environment Markov chain, and $\mu_i \leq \infty$ is the mean number of offspring produced by a single individual if the environment is in state $i$. Our first theorem generalizes the survival/extinction dichotomy for classical branching processes. Let $\mathcal{S} = \{X_t > 0 \text{ for all } t\}$ be the event of survival, where $X_t$ denotes the population of the BPME after $t$ individuals have reproduced. Write $\mathbb{P}_{n.i}$ for the law of the BPME started from population $n$ and environment $i$.
%$(X_t,Q_t)_{t \geq 0}$ given $X_0.Q_0 = n.i$.

\begin{theorem}[Survivial/Extinction] 
\phantomsection
\label{t.survival.intro} ~
\begin{enumerate}
\item If $\mu<1$, then the BPME 
%started in any state 
goes extinct almost surely: $\Prsub{n.i}{\mathcal{S}}=0$ for all $n \in \N$ and all $i \in S$.
%That is, almost surely there exists $t$ such that $X_t = 0$.

\item If $\mu=1$, and the number of offspring produced before the first return to the starting state has positive or infinite variance, then the BPME goes extinct almost surely: $\Prsub{n.i}{\mathcal{S}}=0$ for all $n \in \N$ and all $i \in S$.

\item If $\mu>1$, then the BPME with sufficiently large initial population survives forever with positive probability: For each environment $i$ there exists $n$ such that $\Prsub{n.i}{\mathcal{S}}>0$.
\end{enumerate}
\end{theorem}

%See Theorem~\ref{thm:subcritical precise} for a more precise statement of this theorem.

Our next result gives the asymptotic growth rate for the population on the event of survival. Note that $X_t$ denotes the population after $t$ individuals have reproduced (\emph{not} the population after $t$ generations), so the growth is linear rather than exponential in $t$.

\begin{theorem}[Asymptotic Growth Rate] \label{t.SLLN.intro}
If each offspring distribution has finite variance, then for any initial population $n \geq 1$ and any initial environment $i$,
	\[\frac{X_t}{t} \to (\mu-1)\onesurv \qquad \mathbb{P}_{n.i}\text{-almost surely as } t \to \infty\]
where $\surv =\{X_t >0 \text{ for all } t\}$ is the event of survival.
\end{theorem}

Our next result is a central limit theorem for the normalized population on the event of survival.

\begin{theorem}[Central Limit Theorem]
\label{t.CLT.intro}
If each offspring distribution has finite variance, then for any initial population $n \geq 1$ and any initial environment,
\[ \frac{X_t - (\mu-1) t \onesurv}{\sqrt{t} }\Rightarrow \chi \onesurv  \qquad \text{ as } t \to \infty\]
where $\chi \sim \mathcal{N}(0,\sigma^2_{M})$ is a normal random variable independent of the event of survival $\mathcal{S}=\{X_t > 0 \text{ for all } t\}$. The variance $\sigma^2_{M}$ is computed in Lemma~\ref{lemma: L2 martingale}.
\end{theorem}

To state our last result, we define the \textbf{matrix generating function}
	\[f(M)  = \sum_{n \geq 0} P_n M^n \]
where $M$ is an $S \times S$ substochastic matrix, and $(P_n)_{ij}$ is the probability that the environment transitions from state $i$ to state $j$ while producing $n$ offspring. Here we interpret $M^0$ as the identity matrix.
The \textbf{extinction matrix} $E$ is the $S \times S$ matrix whose $(i,j)$ entry is the probability that the BMPE started with population $1$ in state $i$ goes extinct in state $j$.  

\begin{theorem}[Extinction Matrix] \label{t.extinction.intro}
The $(i,j)$ entry of $E^n$ is the probability that the BPME started with population $n$ in state $i$ goes extinct in state $j$.  Moreover, $f(E)=E$, and
	\[ \lim_{n \to \infty} f^n(M) = E \]
for any matrix $M$ satisfying $0 \leq M \leq E$ entrywise.
\end{theorem}

Here $f^n$ denotes the $n$th iterate of $f$. 

\subsection{Related work}

If the environment has only one state, then BPME reduces to a time change of the classical Galton-Watson branching process \cite{Harris,Athreya}. But BPME with multiple states is entirely different from the classical \textbf{multitype branching process} \cite{KS,conceptual}. In the former, the population is scalar-valued (all individuals are identical) and the offspring distribution depends on the state of the environment, while in the latter, the population is vector-valued and the offspring distribution depends on the type of individual. 

Branching processes have been generalized in various ways, but not, to our knowledge, in the manner of the present paper. 
Athreya and Karlin \cite{AK} consider branching in a stationary random environment. In their setup, the environment changes after each generation instead of each reproduction event. Their analysis features the composition of a random sequence of univariate generating functions, instead of iteration of a single matrix generating function.
Jones \cite{age} considers multitype branching in which offspring are produced not only by the current generation but a fixed number of previous generations, each with its own offspring distribution.

Like BPME, the \textbf{batch Markovian arrival process} (BMAP) \cite{batch} consists of population-environment pairs, but the offspring distribution does not depend on the state of the environment; instead, the wait time between reproduction events depends on the environment, and the probability of transitioning from environment $i$ to environment $j$ depends on the number of offspring produced. A matrix generating function also figures prominently in the analysis of the BMAP, which focuses on queuing applications instead of the survival probabilities and limit theorems studied in this paper.

A unary \textbf{stochastic abelian network} \cite{an1} can be viewed as a multitype BPME. In this context, survival corresponds to the nonhalting networks studied in \cite{an4}. It would be interesting to extend the results of this paper to stochastic abelian networks.

\section{Formal definition and an example}\label{sec: notation}

\subsection{The Markovian environment}\label{sec: BPME definition}

Let $P$ be the transition matrix of an irreducible Markov chain on a finite state space $S$.
The entry $P(i,j)$ is the probability of transitioning from state $i \in S$ to state $j \in S$. We call this Markov chain the \textbf{environment}. We associate to each state $i \in S$ an offspring distribution.
 These offspring distributions can be simultaneously described by a stochastic matrix $R: \R^S \to \R^{\N}$ called the \textbf{reproduction matrix}, where $R_{in}$ is the probability that an individual has $n$ offspring given that the current environment is $i$.  
 %The environment is global (i.e., the same for all individuals) and it updates with each reproduction event.

Let $(\xi_t^{i})_{i \in S, \, t\in \N}$ be independent random variables such that $\Pr{\xi_t^i=n} = R_{in}$ for all $i \in S$ and $n \in \N$.  We interpret $\xi_t^{i}$ as the number of offspring produced at time $t$ if the environment chain happens to be in state $i$ at time $t$. 

\begin{definition}[Branching Process In A Markovian Environment (BPME)]\label{def:BPME}
A \emph{branching process in a Markovian environment} is a sequence $(X_t, Q_t)_{t \geq 0}$ of population-state pairs that evolves as follows. We begin with initial population $X_0 \in \N_{\geq 1}$ and initial state $Q_0 \in S$. Given $X_t$ and $Q_t$, we update:
\begin{align}
\begin{split} \label{eq:update rule}
	Q_{t+1} &\sim P(Q_{t},\cdot) \text{ independent of $X_0,\dots,X_t,Q_0,\dots,Q_{t}, \text{ and } (\xi_t^i)_{t \in \N,i \in S}$}\\
	X_{t+1} &:= \begin{cases} X_{t}-1+\xi_{t+1}, & X_t>0  \\ 0, &X_t=0\end{cases}
\end{split}
\end{align}
where 
\[\xi_{t+1} := \sum_{i \in S} \xi_{t+1}^{i}\ind\{Q_{t+1} = i\}. \]
\end{definition}

That is, the global state $Q_t$ is updated to $Q_{t+1}$ by taking one step in the environment Markov chain. If the population is positive, then one individual produces a random number of offspring $\xi_{t+1}$ sampled from the offspring distribution of the new state $Q_{t+1}$. These new individuals are added to the population, while the reproducing individual is removed from the population.  The integer $X_t$ represents the total population after $t$ such reproduction events. 

In the branching process literature, time is usually indexed by the number of generations.
%Between time $t$ and time $t+1$, every individual alive at time $t$ is replaced by a random number of offspring. 
But in our notation, the time index $t$ is number of reproduction events so far. That is, between time $t$ and time $t+1$, a single individual reproduces (assuming the population has not yet reached zero). This difference in indexing is convenient since the environment updates every time an individual reproduces.

The sequence $(X_t)_{t \geq 0}$ alone is not a Markov chain, but $((X_t,Q_t))_{t \geq 0}$ is a Markov chain. 
% (and $((Q_t,\xi_t))_{t \geq 0}$ is also Markov chain)
A state of the latter chain is an ordered pair $(n,i)$ where $n \in \N$ and $i \in S$. We call such pairs \textbf{total states}. Following \cite{an1} we adopt the notation $n.i$ for the total state $(n,i)$.

Let
	\[ \mu_i :=\sum_{n=0}^\infty nR_{in} \]
be the mean of the offspring distribution of state $i$. Let $\pi_i$ be the stationary probability of environment $i$, and let
\begin{equation}\label{eq:mu}
\mu := \sum_{i \in S} \mu_i \pi_i
\end{equation}
be the mean number of offspring produced when a single individual reproduces in the stationary environment.
This value $\mu$ will play a role analogous to the mean of the offspring distribution in the ordinary Galton-Watson branching process.

\subsection{An Example}
\label{ex: supercritical}

The following example illustrates why Theorem~\ref{t.survival.intro}(3) requires a sufficiently large starting population for the BPME to have positive probability of surviving forever.

Let the environment chain be $S=\{a,b\}$ with $P_{ab}=P_{ba}=1$. Its stationary distribution is $\pi_a=\pi_b = \frac{1}{2}$. Let the offspring distribution of state $a$ be uniform on $\{0,1,2,3,4,5\}$, and let the offspring distribution of state $b$ be $0$ with probability $1$. Then
	$\mu = \mu_a\pi_a + \mu_b\pi_b = 5/4.$ 
The matrix generating function is given by
\[f(M)= \begin{pmatrix} 0 & 1 \\ \frac{1}{6} & 0  \end{pmatrix}+ \begin{pmatrix}0 & 0 \\ \frac{1}{6} & 0 \end{pmatrix}(M+M^2+M^3+M^4+M^5).\]
The extinction matrix $E$ solves the degree $5$ polynomial equation $f(E)=E$.
Using Theorem~\ref{t.extinction.intro}, we can estimate it by calculating $f^n(O)$ for large $n$:
	\[E \approx \begin{pmatrix}  0 & 1 \\ 0.2459 & 0.3497\end{pmatrix}.\]
The BPME started in state $a$ with initial population $X_0=1$ goes extinct immediately (in state $b$), despite the fact that $\mu>1$. However, the extinction matrix for initial population $X_0=2$ is given by:
	\[E^2 \approx \begin{pmatrix} 0.2459 & 0.3497 \\ 0.0860 & 0.3681 \end{pmatrix}. \]
Since both row sums are $<1$, this BPME with initial population $2$ has positive probability to survive forever in either starting state.	

In this example, Theorem~\ref{t.SLLN.intro} tells us that on the event of survival, this BPME population satisfies $X_t / t \to 1/4$ almost surely. Theorem~\ref{t.CLT.intro} tells us that on the event of survival, the population $X_t$ is asymptotically normal with mean $\frac14 t$ and variance $\frac{35}{24}t$.

\old{%
\begin{figure}
\centering
\includegraphics[width = 4in]{simul.png}
\caption{Ten instances of the BPME from Example~\ref{ex: supercritical} started in state $a$ with population $X_0=2$. }
\label{fig:simulation}
\end{figure} 
}

\subsection{Review of Markov Chains}

We recall a few facts about Markov chains, which we will apply in later sections to the environment chain.  Let $(Q_t)_{t \in \N}$ be an irreducible Markov chain with finite state space and transition matrix $P$.
% (so $P_{ij}$ is the probability of transitioning from state $i$ to state $j$ in one step). 
Denote by $\PP_i$ and $\EE_i$ the probability and expectation given $Q_0=i$, and let $\tau_i := \min\{t \geq 1 \mid Q_t=i\}$. 

\begin{proposition}
\phantomsection
\label{MC}
\begin{enumerate}
\item[(i)] There is a unique probability row vector $\pi$ such that $\pi P=\pi$. Moreover, $\pi_i>0$ for all states $i$. 
\item[(ii)] For all $i,j \in S$
	\[ \PP_i \left( \lim_{t \to \infty} \frac{1}{t} \sum_{s=1}^t \ind\{Q_s = j\} = \pi_j \right) = 1. \]
\item[(iii)] $\Esub{i}{\tau_i }=\frac{1}{\pi_i}<\infty$. %In particular, $\tau_i$ is almost surely finite.
\item[(iv)] If $N_j:=\sum_{t=1}^{\tau_i} \ind\{Q_t=j\}$ is the number of visits to $j$ until hitting $i$, then $\Esub{i}{N_j} = \pi_j\Esub{i}{\tau_i}$. Note that (iii) is the special case obtained by setting $j=i$.
% and noting that $\Esub{i}{N_i}=1$.
\item[(v)] There exist constants $C_0,C_1>0$ such that for all $t \geq 0$ and all states $i,j$,
	\[\PP_i (\tau_j>t) \leq C_0e^{-C_1t}.\]
\end{enumerate}
\end{proposition}

\begin{proof}
\begin{enumerate}
\item[(i)] See~\cite{LPW}, Proposition 1.14 (i) and Corollary 2.17.
\item[(ii)] This follows from the Birkhoff ergodic theorem.  
% not a stationary sequence unless we start stationary, but we reach a stationary state in finite time.
\item[(iii)] See~\cite{LPW} Proposition 1.14 (ii).
\item[(iv)] This follows from~\cite{LPW} Lemma 10.5. 
\item[(v)] See~\cite{AF} Section 2.4.3
\end{enumerate}
\end{proof}

\section{The $\Z$-Valued Process }

The proof of Theorem~\ref{t.survival.intro} will proceed by comparing the BPME to a process whose population is allowed to become negative. This \textbf{$\Z$-valued BPME} is a sequence $(Y_t,Q_t)_{t \geq 0}$ of population-state pairs, with initial state $Q_0 \in S$, but now the initial population $Y_0$ is allowed to take values in $\Z$, and reproduction occurs regardless of whether the population is positive. Using the same definitions and notation as in Definition~\ref{def:BPME}, the update rule is given by:
\begin{align}
\begin{split} \label{eq:negative update rule}
	Q_{t+1} &\sim P(Q_{t},\cdot) \text{ independent of $X_0,\dots,X_t,Q_0,\dots,Q_{t-1}$}\\
	Y_{t+1} & := Y_{t}-1+\xi_{t+1}
\end{split}
\end{align}
where 
\[\xi_t := \sum_{i \in S} \xi_t^{i} \ind\{Q_{t+1} = i\}. \]
Notice that we can recover the original BPME from the $\Z$-valued BPME with $Y_0 \geq 1$ by setting:
\[X_t = \begin{cases} Y_t & \text{if } Y_s >0 \text{ for all } 0 \leq s < t \\ 0 & \text{else.} \end{cases}\]
Note that $X_t>0$ for all $t$ if and only if $Y_t>0$ for all $t$.

\subsection{Excursions of the environment}

The proof of Theorem~\ref{t.survival.intro}  
will proceed by considering excursions of the environment chain from its starting state in the $\Z$-valued BPME. 

Fix a starting environment $Q_0=i$. Let $\tau_0=0$, and for $n \geq 1$ let
\begin{align*}
 \tau_{n}&:=\inf\{t > \tau_{n-1} \mid Q_t = i\}.
\end{align*}
be the time of $n$th return to state $i$.

Let 
\[ \Delta_n = \Delta_n^{i} = Y_{\tau_n}-Y_{\tau_{n-1}} .\]
be the net population change during the $n$th excursion from state $i$. 

\begin{lemma}\label{lemma: Delta_n IID}
The sequence $(\Delta_n)_{n \geq 1}$ is independent and identically distributed (i.i.d.).
\end{lemma}

\begin{proof}
We have
\[\Delta_n =\sum_{t = \tau_{n-1}+1}^{\tau_n} (\xi_t-1) = \sum_{t = \tau_{n-1}+1}^{\tau_n} \sum_{j \in S}(\xi_t^j-1)\ind\{Q_t = j\}.\]
By the strong Markov property, the sequence $(Q_{\tau_{n-1}+1}, \dots,Q_{\tau_n})$ has the same distribution for each $n$, and is independent of $(Q_0,\dots,Q_{\tau_{n-1}})$. In addition, the stacks $(\xi_t^j)_{t \geq 0}$ are independent of the stopping times $\tau_n$. Hence, for fixed $j$, the sequence $(\xi_{\tau_{n-1}+1}^j, \dots,\xi_{\tau_n}^j)$ has the same distribution for each $n$, and is independent of every other such sequence for varying $n$ and $j$. It then follows from the expression for $\Delta_n$ above that $(\Delta_n)_{n \geq 1}$ is an \IID sequence. 
\end{proof}

Lemma~\ref{lemma: Delta_n IID} implies that the sequence $(Y_{\tau_n})_{n \geq 0}$ is a random walk with \IID steps $\Delta_1,\Delta_2,\dots$. The long-term behavior of such a walk is determined by the mean of the step distribution. 

\begin{proposition}[{\cite[Exercise 5.4.1]{Durrett5}}]\label{prop: RW}
Let $\Delta_1,\Delta_2,\dots$ be \IID real-valued random variables with finite mean, and $Y_n = Y_0 + \sum_{i=1}^n \Delta_i$. There are only four possibilities, one of which has probability one.
\begin{enumerate}
\item If $\Pr{\Delta_1=0}=1$, then $Y_n = Y_0$ for all $n$.
\item If $\E{\Delta_1}>0$, then $Y_n \to \infty$.
\item If $\E{\Delta_1}<0$, then $Y_n \to -\infty$.
\item If $\E{\Delta_1}=0$ and $\Pr{\Delta_1=0}<1$, then $-\infty = \liminf Y_n < \limsup Y_n = \infty$.
\end{enumerate}
\end{proposition}
%What happens if $\Delta_1^+$ and $\Delta_1^-$ both have infinite mean? This case doesn't occur for finite Markov chains, but it does for null recurrent chains!

Therefore we need to calculate the expectation of $\Delta_1$. Recall that $\mu = \sum_{i \in S}\mu_i \pi_i$ is the mean number of offspring produced when a single individual reproduces in the stationary environment, and $\tau_1$ is the time of first return to the starting environment.

\begin{lemma}\label{l.mean}
The mean number of offspring produced in one excursion from the starting state is $\E{\Delta_1} = \E{\tau_1}(\mu-1)$. 
\end{lemma}

\begin{proof}
Define $\mathcal{F}_t = \sigma(Y_0,\dots,Y_{t-1},Q_0,\dots,Q_{t})$. Note the inclusion of $Q_{t}$: this sigma field includes all information up to the time right before the $t$th individual reproduces. Then,
\begin{align*}
	\E{\Delta_n} &= \sum_{t=\tau_{n-1}+1}^{\tau_n}\sum_{j \in S} \E{\Cond{(\xi_t^j -1)\ind\{Q_t = j\}}{ \mathcal{F}_t}}\\
	&=\sum_{t=\tau_{n-1}+1}^{\tau_n} \sum_{j \in S} \E{\ind\{Q_t=j\} \Cond{(\xi_t^j -1)}{ \mathcal{F}_t}}\\
	&=\sum_{t=\tau_{n-1}+1}^{\tau_n} \sum_{j \in S} \E{\ind\{Q_t=j\} (\mu_j-1)}\\
	&= \sum_{j \in S}(\mu_j-1) \E{\sum_{t=\tau_{n-1}+1}^{\tau_n} \ind\{Q_t=j\}}.
\end{align*}
Now by the strong Markov property and Proposition~\ref{MC}(iv), we have
\[\E{\sum_{t=\tau_{n-1}+1}^{\tau_n} \ind\{Q_t=j\}} = \pi(j)\E{\tau_n - \tau_{n-1}}.\]
By the strong Markov property we also have $\E{\tau_n-\tau_{n-1}} = \E{\tau_1}$. Thus,
\begin{align*}
\E{\Delta_n}&=\sum_{j \in S}(\mu_j-1)\pi(j)\E{\tau_1}\\
&=\E{\tau_1} \left(\sum_{j \in S} \mu_j \pi(j) - \sum_{j \in S} \pi(j)\right)\\
&= \E{\tau_1}(\mu-1). \qedhere
\end{align*}
\end{proof}

\subsection{Extinction in the subcritical and critical cases}

We now prove items (1) and (2) of Theorem~\ref{t.survival.intro}.

\begin{theorem}\label{thm:subcritical precise}
Let $(X_t,Q_t)_{t \geq 0}$ be a BPME with any initial population $X_0$ and any initial environment $Q_0$. 

If $\mu<1$, then $X_t \to 0$ almost surely.

If $\mu=1$ and $0<\Var(\Delta_1) \leq \infty$, 
then $X_t \to 0$ almost surely.
\end{theorem}

\begin{proof}
The $\Z$-valued process $Y_{\tau_n} = Y_0 + \sum_{k=1}^n \Delta_k$, at time of $n$th return to the initial environment $Q_0$, is a random walk with \IID increments $\Delta_k$. By Lemma~\ref{l.mean}, we have $\E{\Delta_k} = \E{\tau}(\mu-1)$. In the case $\mu < 1$, we are in case (3) of Proposition~\ref{prop: RW}; in the case $\mu=1$ and $\text{Var}(\Delta_1)>0$ we are in case (4). So in either case, 
	\[ \liminf_{n \to \infty} Y_{\tau_n} = -\infty \qquad \as. \] 
So with probability one there exists a time $T$ such that $Y_T\leq 0$.  
Since $Y$ is integer-valued and decreases by at most one at each time step, for the minimal such $T$ we have $Y_T=0$, and so $X_{t}=0$ for all $t \geq T$. 
\end{proof}

\subsection{Survival in the supercritical case}

The proof of Theorem~\ref{t.survival.intro}(3) will also proceed by studying $\Delta_n$, the net population change during the $n$th excursion from the starting state in the $\Z$-valued BPME. If we apply Proposition~\ref{prop: RW} to the case $\E{\Delta_1} >0$, we find that $Y_{\tau_n} \to \infty$ almost surely. However, this does not imply $X_t>0$ for all $t$ almost surely, or even with positive probability: We could have chosen a starting state that dooms the process to extinction in the first step, as in Example~\ref{ex: supercritical}. To rule out this kind of scenario, we make the following definition, in which $\mathbb{P}_{m.i}$ denotes the law of the $\Z$-valued BPME started with $(Y_0,Q_0)=(m,i)$, and $\tau:=\inf\{t \geq 1 \mid Q_t=i\}$ is the time of first return to state $i$.  
\begin{definition}
Total state $m.i$ is \textbf{viable} if 
\[\Prsub{m.i}{Y_{\tau}-Y_0 \geq 1 \text{ and } Y_t  \geq 1 \text{ for all } t \in [0,\tau]}>0.\]
\end{definition}
In words, $m.i$ is viable if it is possible to start in state $i$ with $m$ individuals and return to state $i$ with at least $m+1$ individuals, while keeping the population positive the whole time. Note that if $m.i$ is viable then $(m+1).i$ is viable.

\begin{lemma}
\label{l.viable}
Suppose $\mu>1$. Then for every state $i \in S$, there exists $m$ such that $m.i$ is viable. 
\end{lemma}

\begin{proof}
Fixing $i \in S$, let
$\mathrm{Cyc}$ be the set of all tuples $(y_0.q_0,\dots,y_n.q_n)$ (of any length) with each $y_t \in \mathbb{Z}$ and each $q_t \in S$, such that
$y_0 = 0$ and $q_0=q_n=i$ and $q_t \neq i$ for all $1\leq t<n$.
Let
	\[\Pr{ (y_0.q_0,\dots,y_n.q_n) }:= \Prsub{y_0.q_0}{Y_t.Q_t = y_t.q_t \text{ for all } 1 \leq t \leq n }. \]
The mean population change over one excursion from state $i$ is
	\[ \Esub{0.i}{\Delta_1 }= \sum_{c \in \mathrm{Cyc}} y_n \Pr{c}. \]
Since $\mu>1$ we have $\Esub{0.i}{\Delta_1} = \E{\tau}(\mu-1)>0$ by Lemma~\ref{l.mean}, so at least one term on the right side is positive. Hence there exists $c = (0.q_0,\dots,y_n.q_n) \in \mathrm{Cyc}$ such that $y_n \geq 1$ and $\Pr{c}>0$.  Since $y_0=0$ and the population can decrease by at most one per time step, each $y_i \geq -i$. So the excursion 
	\[ c' = ((y_0+n).q_0,\dots,(y_n+n).q_n) \]
has $\Pr{c'} = \Pr{c} >0$ and all $y_i + n \geq 1$ and $y_n + n \geq n+1$, so $n.i$ is viable.
\end{proof}

\begin{remark}
A similar argument shows that if no state's offspring distribution is concentrated on $0$, then $1.i$ is viable for all states $i$.
\end{remark}

We are now ready to prove the main result of this section. 

\begin{theorem}
\label{t.supercritical}
Let $X_t$ be a BPME with $\mu>1$. Then for all viable $m.i$ we have
	\[ \Prsub{m.i}{X_t > 0 \text{ for all } t} >0. \]
\end{theorem}

\begin{proof}
Let $m.i$ be viable. Write $\mathbb{P} = \mathbb{P}_{m.i}$ for the law of the $\Z$-valued the BPME $(Y_t,Q_t)_{t \geq 0}$, started with initial population $m$ and initial state $i$. Since $X_t>0$ for all $t$ if and only if $Y_t>0$ for all $t$, it suffices to prove $\Pr{Y_t>0 \text{ for all } t }>0$. 

Since $m.i$ is viable, there exists $\delta$ such that
\[\Pr{Y_\tau \geq m+1 \text{ and } Y_t \geq 1 \text{ for all } t \in (0,\tau]}\geq \delta >0.\]
By the strong Markov property and induction on $n$ it follows that
	\[\Pr{Y_{\tau_n} \geq m+n \text{ and } Y_t \geq 1 \text{ for all } t \in (0,\tau_n]} \geq \delta^n\]
for all $n \geq 1$.  
Write $\mathcal{E}_n = \{Y_{\tau_n} \geq m+n\}$. By the strong Markov property at time $\tau_n$,
\begin{align*}
\Pr{Y_t>0 \text{ for all } t } 
&\geq \Pr{Y_t>0 \text{ for all } t > \tau_n \mid \mathcal{E}_n \text{ and } Y_t \geq 1 \text{ for all } t \in (0,\tau_n]} \cdot \delta^n \\
&= \Pr{Y_t>0 \text{ for all } t > \tau_n \mid \mathcal{E}_n} \cdot \delta^n.
\end{align*}
So it suffices to find $n$ such that the right side is strictly positive, or equivalently,
	\[\Pr{Y_t = 0 \text{ for some } t > \tau_n \mid \mathcal{E}_n}<1.\]
	
Consider the following events for $k \geq 0$:
\begin{align*}
	A_k&:= \{Y_t = 0 \text{ for some } t \in (\tau_k, \tau_{k+1}]\}\\
	B_k&:= \left\{Y_{\tau_k} \leq \frac{\mu-1}{2}\E{\tau}k\right\}\\
	C_k&:= \left\{\tau_{k+1}-\tau_k \geq \frac{\mu-1}{2}\E{\tau}k\right\}
\end{align*}
where $\tau = \tau_1$. Observe that 
\[	\{Y_t = 0 \text{ for some } t > \tau_n\} = \bigcup_{k=n}^\infty A_k\]
and for each $k$,
\[ A_k \subseteq B_k \cup C_k.\]
The latter holds because $Y_t$ decreases by at most one at each time step; hence if $Y_{\tau_k}> \frac{\mu-1}{2}\E{\tau}k$, then $Y_t$ cannot reach $0$ before time $\tau_{k+1}$ unless $\tau_{k+1}-\tau_k \geq \frac{\mu-1}{2}\E{\tau}k$.
We wish to show
	\[\Pr{\bigcup_{k=n}^\infty A_k \mid \mathcal{E}_n} <1.\]

Recall from Lemmas~\ref{lemma: Delta_n IID} and~\ref{l.mean} that $Y_{\tau_k} = Y_0 + \sum_{m=1}^k \Delta_m$, where the $\Delta_m$ are \IID with mean $\E{\tau}(\mu-1)$. By the strong law of large numbers, $\frac{1}{k}Y_{\tau_k} \to \E{\tau}(\mu-1)$ almost surely as $k \to \infty$. In particular, this implies $\Pr{B_k \text{ i.o.}} = 0$, so, $\Pr{\bigcup_{k=n}^\infty B_k} \to 0$ as $n \to \infty$. By the FKG inequality,
%(See Lemma~\ref{lemma: FKG})
since $B_k$ is a decreasing event and $\mathcal{E}_n$ is an increasing event with respect to the offspring random variables $(\xi_t^j)_{j \in S, \, t \in \N}$, we have 
\[ \Pr{\bigcup_{k=n}^\infty B_k \mid \mathcal{E}_n} \leq \Pr{\bigcup_{k=n}^\infty B_k}. \]

Since each $\tau_{k+1}-\tau_k$ has the same distribution as $\tau$, we have
	\[  \sum_{k \geq 0} \Pr{C_k} < \infty. \]	
In addition, $C_k$ is independent of $\mathcal{E}_n$ for all $k \geq n$ (since $\tau_{k+1}-\tau_k$ is independent of $Y_{\tau_n}$ by the strong Markov property). 

Now choose $n$ large enough so that $\Pr{\bigcup_{k=n}^\infty B_k}< \frac{1}{2}$ and $\sum_{k=n}^\infty \Pr{C_k} <\frac{1}{2}$. Then
\begin{align*}
	\begin{split}
	\Pr{\bigcup_{k=n}^\infty A_k \mid \mathcal{E}_n }& \leq \Pr{\bigcup_{k=n}^\infty B_k \mid \mathcal{E}_n} 
	+ \Pr{\bigcup_{k=n}^\infty C_k \mid \mathcal{E}_n} 
	\end{split}\\
	&\leq \Pr{\bigcup_{k=n}^\infty B_k} + \Pr{\bigcup_{k=n}^\infty C_k} \\
	&< \frac{1}{2}+\frac{1}{2}
\end{align*}
where we have used the FKG inequality to remove the conditioning on the $B$ term, and independence to remove the conditioning on the $C$ term.
\end{proof}

Theorem~\ref{t.survival.intro}(3) follows immediately from Theorem~\ref{t.supercritical} together with Lemma~\ref{l.viable}.

\section{Laws of large numbers}

In this section we prove strong laws of large numbers for $Y_t$ and $X_t$. As above, the environment chain $(Q_t)_{t \geq 0}$ is irreducible with stationary distribution $\pi$, and for each state $i$ the offspring distribution has finite mean $\mu_i$, and we let
	$ \mu = \sum_{i \in S} \pi_i \mu_i. $

\begin{theorem}\label{t.SLLN}
For any initial populations $Y_0 \in \Z$ and $X_0 \in \N_{ \geq 1}$, it holds almost surely as $t \to \infty$
	\begin{align*}
	\frac{Y_t}{t} \to  \mu-1
	\end{align*}
and
	\begin{align*}
	\frac{X_t}{t} \to (\mu-1)\ind_{\mathcal{S}} 
	\end{align*}
where $\mathcal{S}=\{X_t \geq 1 \text{ for all } t\}$ is the event of survival.
\end{theorem}

\begin{proof}
We start by observing that
	\[ Y_t - Y_0 = \sum_{s=1}^t (\xi_s-1) = \sum_{s=1}^t \sum_{i \in S} \ind \{Q_s=i \} (\xi_s^i-1). \]
Switching the order of summation, the right side can be written as
	\[ \sum_{i \in S} Z^i_{L_t} \]
where $Z^i_k$ denotes a sum of $k$ independent copies of $\xi^i_1-1$, and $L_t = \sum_{s=1}^t \ind \{Q_s=i \}$ is the local time of state $i$.

By the strong law of large numbers for i.i.d.\ sums, $Z^i_k / k \to \mu_i - 1$ almost surely as $k \to \infty$. It follows from Proposition~\ref{MC}(ii) that
	\[ \frac{Z^i_{L_t}}{t} = \frac{Z^i_{L_t}}{L_t} \frac{L_t}{t} \to (\mu_i -1) \pi_i \]
almost surely as $t \to \infty$. Summing over $i$ yields $Y_t / t \to \mu -1$ almost surely.

The strong law for $X_t$ now follows by observing that $X_t = Y_t$ on $\mathcal{S}$ and $X_t/t \to 0$ almost surely on $\mathcal{S}^c$.
\end{proof}

\section{Central Limit Theorems}\label{sec: martingales}
% Could be extended to the case $S$ is infinite and positive recurrent, with tail bounds on hitting times.

In this section we prove a central limit theorem for $Y_t$ and $X_t$ under a second moment assumption on the offspring distributions. 
We will proceed by defining a martingale involving $Y_t$, calculating its quadratic variation, and invoking the martingale central limit theorem. Then to pass from $Y_t$ to $X_t$, we show that the limiting normal random variable is independent of the event of survival.

\subsection{The martingale}\label{subsec: martingale}

As above, suppose the environment chain $(Q_t)_{t \geq 0}$ is irreducible with transition matrix $P$ and stationary distribution $\pi$, and for each state $i$ the offspring distribution has finite mean $\mu_i$ (We will impose finite second moment in the next section, but it is not needed yet). 
Write $\vec{\pi}$ and $\vec{\mu}$ for the column vectors with coordinates $\pi_i$ and $\mu_i$ respectively. As above, let
	\[ \mu = \langle \vec{\pi}, \vec{\mu} \rangle := \sum_{i \in S} \pi_i \mu_i. \]

\begin{lemma}\label{claim: state function}
There is a unique vector $\vec{\varphi} \in \R^S$ satisfying $\left\langle \vec{\pi},\vec{\varphi}\right\rangle = 0$ and
		\[ \varphi_i - \sum_{j \in S} P_{ij} \varphi_j  = \mu_i-\mu \]
for all $i \in S$.
\end{lemma}

\begin{proof}
Since $\vec{\pi}$ is the unique stationary distribution of the environment chain, the left null space of $I-P$ is spanned by $\vec{\pi}$. The column space of $I-P$ is the orthogonal complement of the left null space,  
\[\text{Im}(I-P) = \{\vec{v} \in \R^S : \left\langle \vec{\pi}, \vec{v} \right\rangle = 0\}.\]
We check that $\mu\vec{1}-\vec{\mu}$ is orthogonal to $\vec{\pi}$, and thus is in the image of $I-P$:
\begin{align*}
\sum_{i \in S} \pi_i (\mu-\mu_i) 
&=\mu \sum_{i \in S} \pi_i - \sum_{i \in S} \pi_i\mu_i \\
&=\mu-\mu\\
&=0.
\end{align*}
Therefore there exists a vector $\vec{\varphi}$ such that $(I-P)\vec{\varphi}=\vec{\mu}-\mu\vec{1}$ as desired.

The right null space of $I-P$ is spanned by $\vec{1}$, so $\vec{\varphi}$ is unique up to adding scalar multiples of $\vec{1}$.  Therefore there is a unique such vector satisfying $\left\langle \vec{\pi},\vec{\varphi} \right \rangle = 0$.
\end{proof}

We will see that the vector $\vec{\varphi}$ has a natural interpretation in terms of the BPME: Its coordinate $\varphi(i)$ represents the \textbf{long-term excess fertility} of environment $i$, in the sense of Corollary~\ref{c.fertility} below.

We now define a martingale for the $\Z$-valued branching process $(Y_t,Q_t)_{t \geq 0}$, adapted to the filtration 
\[\mathcal{F}_t := \sigma((\xi^i_s)_{i \in S, 0 \leq s \leq t},(Q_s)_{0 \leq s \leq t+1}).\] 
Note the inclusion of $Q_{t+1}$: This sigma-algebra tells us what state we will transition to next, but not how many offspring will be produced.

\begin{lemma}\label{claim: martingale}
Write $\varphi(i)$ for the $i^{th}$ coordinate of the vector $\vec{\varphi}$ of Lemma~\ref{claim: state function}. Then
	\[M_t:=Y_t - (\mu-1) t + \varphi(Q_{t+1})\]
is a martingale adapted to $\mathcal{F}_t$.
\end{lemma}

Note that $Y_t - \sum_{s=1}^t (\mu_{Q_s} -1)$ is also a martingale, but we will find $M_t$ much more useful!  % For a quenched CLT both martingales will be useful.

\begin{proof}
Recall that $Y_{t}=Y_{t-1}-1+\xi_t$ where $\xi_t=\sum_{i \in S}\xi_t^i \ind\{Q_t=i\}$ and $\xi_t^i$ is sampled from the offspring distribution of state $i$. Hence
\begin{align*}
\Cond{M_{t}}{\mathcal{F}_{t-1}} 
%&= \Cond{Y_{t} -(\mu-1)t + \varphi(Q_{t+1})}{\mathcal{F}_{t-1}}\\
%&=\Cond{Y_{t-1} -1 + \xi_{t} - (\mu-1)t + \varphi(Q_{t+1})}{\mathcal{F}_{t-1}}\\
&= Y_{t-1} -1- (\mu-1)t + \Cond{\xi_{t} + \varphi(Q_{t+1})}{\mathcal{F}_{t-1}}.
\end{align*}
Now
\begin{align*}
 \Cond{\xi_{t} + \varphi(Q_{t+1})}{\mathcal{F}_{t-1}} 
 %&= \Cond{\sum_{i \in S}\ind\{Q_{t}=i\}(\xi_{t}^i + \sum_{j \in S} \ind\{Q_{t+1}=j\} \varphi(j))}{\mathcal{F}_{t-1}}\\
 &=\sum_{i \in S} \ind\{Q_{t}=i\} \Cond{\xi_{t}^i + \sum_{j \in S} \ind\{Q_{t+1}=j\}\varphi(j)}{\mathcal{F}_{t-1}}\\
 &= \sum_{i \in S} \ind\{Q_{t}=i\}\left(\mu_i + \sum_{j \in S} P(i,j) \varphi(j)\right)\\
 &= \sum_{i \in S} \ind\{Q_{t}=i\}(\mu + \varphi(i))\\
 &= \mu + \varphi(Q_{t})
\end{align*}
where we have used the fact that $\sum_{j \in S} P(i,j)\varphi(j)-\varphi(i) = \mu-\mu_i$. Combining this with the above, we have
\begin{align*}
\Cond{M_{t}}{\mathcal{F}_{t-1}} &= Y_{t-1} -1 - (\mu-1)t + \mu + \varphi(Q_{t})\\
&= M_{t-1}. \qedhere
\end{align*}
\end{proof}

\begin{corollary} \label{c.fertility}
If the environment chain is aperiodic, then
	\begin{equation*} \label{eq:fertility} \varphi(i) = \lim_{t \to \infty} \left( \Esub{0.i}{Y_t} - (\mu-1) t \right). \end{equation*}
\end{corollary} 
%The corollary is true without the aperiodic assumption, but it requires a different proof!

\begin{proof}
Let $Y_0 = 0$ and $Q_0 = i$. Writing $\mathbb{E} = \mathbb{E}_{0.i}$, equating $M_0 = \E{M_t}$ yields
	\[ \varphi(i) = \E{Y_t} - (\mu-1)t - \E{\varphi(Q_{t+1})}. \]
Since the environment is aperiodic, $\Pr{Q_t = j} \to \pi_j$ as $t \to \infty$, and hence
	\[ \E{\varphi(Q_{t+1})} \to \langle \vec{\pi}, \vec{\varphi} \rangle = 0. \qedhere \]
\end{proof}

\old{%
\begin{proof} %direct proof without the martingale:
Writing $f_{t,i} =  \Esub{i}{Y_t}$, we have by conditioning on the first time step
	\[ f_{t,i} = \mu_i-1 + \sum_{j \in S} P(i,j) f_{t-1,j}. \] 
So $\vec{f}_t = -t\vec{1} + \vec{\mu} + P\vec{\mu} + \cdots + P^{n-1} \vec{\mu}$.
Note that $\langle \vec{\pi}, P^j\vec{\mu} \rangle = \langle \vec{\pi}, \vec{\mu} \rangle = \mu$ for all $j \geq 0$. Writing $\vec{g}_t = \vec{f}_t + t(\mu-1)\vec{1}$ we have $\langle \vec{\pi}, \vec{g}_t \rangle = 0$ and
	\[ (I-P)\vec{g}_t  = (I-P^t)\vec{\mu}. \]
The right side converges to $\vec{\mu} - \mu\vec{1}$ as $t \to \infty$ by aperiodicity. Since 
$I-P$ is invertible on the space $\{v \in \R^S \,:\, \langle \vec{\pi}, \vec{v} \rangle = 0\}$, it follows that $\lim_{t \to \infty} \vec{g}_t$ exists and equals $\vec{\varphi}$.
\end{proof}
}%

\subsection{Quadratic Variation}

Assume now that each offspring distribution has mean $\mu_i$ and variance $\sigma_i^2 < \infty$.  
In this case $EM_t^2 < \infty$; to see this, note that $\varphi$ is bounded (since the state space $S$ is finite) and 
	\[ Y_t = Y_0 + \sum_{s=1}^t \sum_{i \in S} \ind \{Q_s=i\} (\xi_s^i -1)\]
and each $\xi_s^i$ is square-integrable, so $M_t$ is a finite sum of square-integrable random variables.
The \textbf{quadratic variation} associated with $M_t$ is
	\[V_t:=\sum_{s=1}^t \Cond{(M_t-M_{t-1})^2}{\mathcal{F}_{t-1}}.\]

\begin{lemma}\label{lemma: L2 martingale}
$V_t/t \to \sigma^2_{M}$ almost surely as $t \to \infty$, where
	\begin{equation} \label{eq:sigma_M^2} \sigma^2_{M}:= 
	\sum_{i \in S} \pi_i \large(\sigma_i^2 - (\mu-\mu_i)^2 + 2\mu_i \varphi_i 
	\large) \end{equation}
where $\vec{\varphi}$ is given by Lemma~\ref{claim: state function}.
\end{lemma}

\begin{proof}

The increments of $M$ are given by
	\[M_s - M_{s-1} = 
	\xi_s -\mu + \varphi(Q_{s+1})-\varphi(Q_s).\]
Squaring and taking conditional expectation, we break the result into three terms:

\begin{align*}
\Cond{(M_s - M_{s-1})^2 }{\mathcal{F}_{s-1}} &= \underbrace{\Cond{(\xi_s-\mu)^2 }{\mathcal{F}_{s-1}}}_{(1)} + \underbrace{\Cond{(\varphi(Q_{s+1})-\varphi(Q_s))^2 }{\mathcal{F}_{s-1}}}_{(2)}\\
&\qquad + \underbrace{2\Cond{(\xi_s-\mu)(\varphi(Q_{s+1})-\varphi(Q_s))}{\mathcal{F}_{s-1}}}_{(3)}.
\end{align*}

We consider these terms one at a time. First, since $Q_s$ is $\mathcal{F}_{s-1}$-measurable,
\begin{align*}
(1) %=\Cond{(\xi_s-\mu)^2 }{\mathcal{F}_{s-1}} 
&=\sum_{i \in S}\ind\{Q_s=i\}\Cond{(\xi_s^i-\mu)^2}{\mathcal{F}_{s-1}}.
\end{align*}
The conditioning on the right side can be dropped by independence. Adding and subtracting $\mu_i$ we obtain
\begin{align*}
 \sum_{i \in S} \ind\{Q_s=i\}\EE ((\xi_s^i - \mu_i) - (\mu-\mu_i))^2 \\
\end{align*}
Taking the time average, since $\frac{1}{t}\sum_{s=1}^t \ind\{Q_s=i\} \to \pi_i$ almost surely as $t \to \infty$, we have
\begin{align}
\frac{1}{t}\sum_{s=1}^t \sum_{i \in S} \ind\{Q_s=i\}(\sigma_i^2 + (\mu-\mu_i)^2)
&\to \sigma^2 + \tau^2
\label{eq:firstterm}
\end{align}
where 
	\[ \sigma^2 := \sum_{i \in S} \pi_i \sigma_i^2 \]
and
	\[ \tau^2 := \sum_{i \in S} \pi_i (\mu - \mu_i)^2. \]	

For the second term,
\begin{align*}
&(2)=\Cond{(\varphi(Q_{s+1})-\varphi(Q_s))^2}{\mathcal{F}_{s-1}} \\
%&\qquad= \sum_{i \in S} \ind\{Q_s=i\}\Cond{ \sum_{j \in S} \ind\{Q_{s+1}=j\} (\varphi(j) - \varphi(i))^2}{\mathcal{F}_{s-1}} \\
&\qquad=\sum_{i \in S} \ind\{Q_s=i\} \sum_{j \in S} P(i,j)(\varphi(j) - \varphi(i))^2 
\end{align*}
Expanding the square and using the definition of $\varphi$, this becomes
\begin{align*}
%&\qquad=\sum_{i \in S} \ind\{Q_s=i\} \sum_{j \in S} P(i,j)(\varphi(j)^2 - 2\varphi(i)\varphi(j)+\varphi(i)^2)\\
&\qquad=\sum_{i \in S} \ind\{Q_s=i\} \left[\sum_{j \in S} P(i,j)\varphi(j)^2 - 2 \varphi(i)\underbrace{\sum_{j \in S} P(i,j)\varphi(j)}_{\mu-\mu_i +\varphi(i)} + \varphi(i)^2\right]\\
&\qquad=
\sum_{i \in S} \ind\{Q_s=i\}\left[ \sum_{j \in S} P(i,j)\varphi(j)^2 - 2 \varphi(i) (\mu-\mu_i) -\varphi(i)^2\right].
\end{align*}
Now taking the time average, and writing $\vec{\varphi^2}$ for the vector whose $i$th entry is $\varphi(i)^2$ and $\vec{\mu} \vec{\varphi}$ for the vector whose $i$th entry is $\mu_i \varphi(i)$, we obtain
\begin{align}
&\frac{1}{t}\sum_{s=1}^t \sum_{i \in S} \ind\{Q_s=i\}\left[ \sum_{j \in S} P(i,j)\varphi(j)^2 - 2\mu \varphi(i)+2\mu_i\varphi(i)-\varphi(i)^2\right] \notag \\
&\to \sum_{i \in S}\pi_i\left[\sum_{j \in S} P(i,j)\varphi(j)^2 - 2\mu \varphi(i)+2\mu_i \varphi(i)-\varphi(i)^2\right] \notag\\
&= \sum_{j \in S} \varphi(j)^2 \underbrace{\sum_{i \in S} \pi_i P(i,j)}_{\pi_j} - 2 \mu \underbrace{\left \langle \vec{\pi},\vec{\varphi}\right\rangle}_{0} + 2\left\langle \vec{\pi},\vec{\mu}\vec{\varphi}\right\rangle - \left \langle \vec{\pi},\vec{\varphi^2}\right\rangle \notag\\
&= 2\left\langle \vec{\pi}, \vec{\mu} \vec{\varphi}\right\rangle.
\label{eq:secondterm}
\end{align}

For the third term, since $\xi_s^i$ is independent of $\sigma(\mathcal{F}_{s-1}, Q_{s+1})$, we have
\begin{align*}
&2\Cond{(\xi_s-\mu)(\varphi(Q_{s+1})-\varphi(Q_s))}{\mathcal{F}_{s-1}} \\
&\qquad= 2\sum_{i \in S} \ind\{Q_s=i\} \Cond{(\xi_s^i-\mu)\left(\sum_{j \in S} \ind\{Q_{s+1}=j\} \varphi(j)-\varphi(i)\right)}{\mathcal{F}_{s-1}}\\
&\qquad=2\sum_{i \in S} \ind\{Q_s=i\}(\mu_i-\mu)\left(\sum_{j\in S}P(i,j)\varphi(j)-\varphi(i)\right)\\
&\qquad= 2\sum_{i \in S} \ind\{Q_s=i\}(\mu_i-\mu)(\mu-\mu_i).
\end{align*}
The limit of the time average is
\begin{align}
-\frac{2}{t}\sum_{s=1}^t \sum_{i \in S}\ind\{Q_s=i\} (\mu_i-\mu)^2  \to -2\tau^2.
\label{eq:thirdterm}
\end{align}

Adding \eqref{eq:firstterm}, \eqref{eq:secondterm}, and \eqref{eq:thirdterm}, we conclude that
\[
\frac{V_t}{t}
\to
\sigma^2 - \tau^2 + 2\left\langle\vec{\pi},\vec{\mu} \vec{\varphi}\right\rangle = \sigma_M^2. \qedhere
\]

\end{proof}

\begin{remark} $\sigma_M^2 \geq 0$ since it is a limit of nonnegative random variables. In particular, taking all $\xi_t^i$ deterministic so that the first term $\sigma^2=0$, we obtain the inequality
	\begin{equation} \label{eq:nonneg} \tau^2 \leq 2\left\langle\vec{\pi},\vec{\mu} \vec{\varphi}\right\rangle. \end{equation}
It would be interesting to give a more direct proof of this algebraic fact. 
Note that if at least one of the offspring distributions has positive variance, then $\sigma^2 >0$ and hence $\sigma_M^2 >0$ by \eqref{eq:nonneg}.  On the other hand, if all offspring distributions are deterministic, then a necessary and sufficient condition for $\sigma_M^2 =0$ is that all excursions from a fixed state have the same net number of offspring.  
To avoid trivialities, we assume from now on that $\sigma_M^2 >0$.
\end{remark}
	
\subsection{Applying the martingale CLT}\label{subsec: martingale CLT}

Our goal in this section is to prove the following central limit theorems for $Y_t$ and $X_t$.

\begin{theorem} \label{t.CLT}
Assume that $\sigma_M^2 >0$. Then we have convergence in distribution
	\[ \frac{Y_t - (\mu-1)t}{\sqrt{t}} \Rightarrow  \chi \]
and
	\[ \frac{X_t - (\mu-1)t \ind_{\mathcal{S}}}{\sqrt{t}} \Rightarrow \chi \ind_{\mathcal{S}} \]
where $\chi$ is a normal random variable with mean $0$ and variance $\sigma_M^2$, and $\chi$ is independent of $\mathcal{S}$, the event of survival.
\end{theorem}

To prove these, we will use the following version of the martingale central limit theorem. As above, let $V_t$ be the quadratic variation associated to the martingale $M_t$. Write $K_t = M_t - M_{t-1}$.
	
\begin{proposition}[Martingale CLT, see {\cite[Theorem 8.2.8]{Durrett5}}]\label{p.MGCLT}
Suppose that as $t \to \infty$
\begin{enumerate}
\item[(i)] $\frac{V_t}{t} \to \sigma_M^2 >0$ in probability, and
\item[(ii)] $\frac{1}{t} \sum_{s=1}^t \E{K_s^2 \ind\{|K_s| > \epsilon\sqrt{t}\}} \to 0$ for all $\eps>0$.
\end{enumerate}
Then $M_t/\sqrt{t} \Rightarrow \mathcal{N}(0,\sigma_M^2)$ as $t \to \infty$.
\end{proposition}

%In fact, under the same hypotheses, the linear interpolation of $M_{tu}/\sqrt{t}$ converges weakly to a Brownian motion.

We start by verifying the above conditions (i) and (ii) for the martingale defined in Lemma~\ref{claim: martingale}.  Condition (i) follows from Lemma~\ref{lemma: L2 martingale}.

To check the Lindeberg condition (ii), let $J$ be a constant such that $|\varphi(i) | \leq J$ for all $i \in S$. Since $\xi_s \geq 0$, 
	\[|K_s| = |\xi_s-\mu+\varphi(Q_{s+1})-\varphi(Q_s)| \leq \xi_s + \mu + 2J.\]
If $t$ is sufficiently large so that $\frac{\epsilon}{2}\sqrt{t} > \mu + 2J$, then:
	\[ \{ |K_s| > \epsilon\sqrt{t} \} \subset \{ \xi_s > \frac{\epsilon}{2}\sqrt{t} \}. \]
In addition, on the event $\xi_s > \frac{\epsilon}{2}\sqrt{t}$, with $t$ large enough that $\frac{\epsilon}{2}\sqrt{t}> \mu + 2J$, we have:
	\[K_s^2 \leq (\xi_s+\mu+2J)^2 < (2\xi_s)^2.\]	
	Hence we have:
	\[\E{K_s^2 \ind\{|K_s| > \epsilon\sqrt{t}\}} \leq\E{K_s^2 \ind\{\xi_s > \frac{\epsilon}{2}\sqrt{t}\}  } \leq \E{4\xi_s^2 \ind\{\xi_s > \frac{\epsilon}{2}\sqrt{t}\}}.\]
	We will show that this quantity goes to $0$ uniformly in $s$ as $t \to \infty$. We have:
	\begin{align*}
	\E{\xi_s^2 \ind\{\xi_s > \frac{\epsilon}{2}\sqrt{t}\}} &= \sum_{i \in S}\E{ \Cond{(\xi_s^i)^2 \ind\{\xi_s^i > \frac{\epsilon}{2}\sqrt{t}, Q_s = i\}}{\mathcal{F}_{s-1}}}\\
	&= \sum_{i \in S} \E{\ind\{Q_s=i\}   \E{(\xi_s^i)^2 \ind\{\xi_s^i > \frac{\epsilon}{2}\sqrt{t}\}}}\\
	&=\sum_{i \in S} \Pr{Q_s=i} \E{(\xi_s^i)^2 \ind\{\xi_s^i > \frac{\epsilon}{2}\sqrt{t}\}}\\
	&\leq \max_{i \in S} \E{(\xi_s^i)^2 \ind\{\xi_s^i > \frac{\epsilon}{2}\sqrt{t}\}}.
	\end{align*}
Now since $\E{(\xi_s^i)^2}< \infty$, we have $\E{(\xi_s^i)^2 \ind\{\xi_s^i > \frac{\epsilon}{2}\sqrt{t}\}} \to 0$ as $t \to \infty$. Moreover, this rate is uniform in $s$ since $(\xi_s^i)_{s \geq 0}$ are \IID samples from the offspring distribution of state $i$. This verifies condition (ii) of the martingale CLT and hence we have shown	\begin{equation} \label{eq:MCLT} M_t/\sqrt{t} \Rightarrow \chi \sim \mathcal{N}(0,\sigma_M^2). \end{equation}

\begin{proof}[Proof of Theorem~\ref{t.CLT}]
Since $(Y_t - (\mu-1)t) - M_t$ is bounded, it follows from \eqref{eq:MCLT} that
	\[ Z_t := \frac{Y_t - (\mu-1)t}{\sqrt{t}} \Rightarrow \chi. \]

To prove the CLT for $X_t$, recall that $X_t = 0$ eventually on $\mathcal{S}^c$, and that $X_t = Y_t$ for all $t$ on $\mathcal{S}$. Applying the CLT for $Y_t$, we find that
	\[ \frac{X_t - (\mu-1)t\ind_{\mathcal{S}}}{\sqrt{t}} \]
converges in distribution to $0$ on $\mathcal{S}^c$, and to $\chi$ on $\mathcal{S}$. It remains to show that $\chi$ is independent of $\mathcal{S}$. To this end, fix an environment state $i$ and let
	\[ s = \inf \{u>t^{1/4} \,:\, Q_u = i\}. \]
(The choice of $t^{1/4}$ is unimportant; any function tending to $\infty$ slower than $\sqrt{t}$ will do.)  By the strong Markov property, the random variable
		\[ Z'_t := \frac{Y_t - Y_s - (\mu-1)t}{\sqrt{t}}. \]
is independent of the event
		\[ \mathcal{S}_s := \{Y_u > 0 \text{ for all } u < s\} \]
of survival up to time $s$.

Note that $s = t^{1/4} + \tau$ for a random variable $\tau$ satisfying $P(\tau>a) \leq \max_j P(\tau_{ji} >a)$ for all $a$, where $\tau_{ji}$ is the first hitting time of state $i$ starting from state $j$. By Proposition~\ref{MC}(v) the hitting times $\tau_{ji}$ have exponential tails, so by Borel-Cantelli we have $s/\sqrt{t} \to 0$ almost surely.  
% In fact we do not need exponential tails here, it suffices that $P(\tau > a) < Ca^{-2-\delta}$.
By Theorem~\ref{t.SLLN} we have $Y_s / s \to \mu-1$ almost surely, so
	\[  \frac{Y_s}{\sqrt{t}} = \frac{Y_s}{s} \frac{s}{\sqrt{t}} \to 0 \]
almost surely, and hence $Z_t - Z'_t \to 0$ almost surely.  Since $\mathcal{S}_s \downarrow \mathcal{S}$, and $Z'_t$ is independent of $\mathcal{S}_s$, for any fixed $\eps>0$ and $a \in \R$ we have for large enough $t$
	\begin{align*} P(Z_t > a, \mathcal{S})  &\leq P(Z_t>a, \mathcal{S}_s) \\
								  &\leq P(Z'_t>a-\eps, \mathcal{S}_s) \\
								  &= P(Z'_t>a-\eps) P(\mathcal{S}_s) \\
								  &\to P(\chi>a-\eps) P(\mathcal{S})
	\end{align*}
as $t \to \infty$. Likewise,
	\begin{align*} P(Z_t > a, \mathcal{S}) &\geq P(Z_t>a, \mathcal{S}_s) - \eps \\
								&\geq P(Z'_t>a+\eps, \mathcal{S}_s) - \eps \\
								&\to P(\chi>a+\eps) P(\mathcal{S}) - \eps.
							\end{align*}
Since $\eps>0$ is arbitrary, we conclude that
	\[ P(Z_t > a, \mathcal{S}) \to P(\chi>a) P(\mathcal{S}). \qedhere \]
\end{proof}

\begin{remark}
Under the same hypotheses as Proposition~\ref{p.MGCLT}, the martingale convergence theorem gives the stronger conclusion that $(M_{tu} / \sqrt{t})_{u \in [0,1]}$ converges weakly on $C[0,1]$ to $(\sigma_M B_u)_{u \in [0,1]}$ where $B$ is a standard Brownian motion. This yields a corresponding strengthening of Theorem~\ref{t.CLT}, namely
	\[ \left(\frac{X_{tu} - (\mu-1)tu \onesurv }{\sqrt{t}} \right)_{u \in [0,1]} \Rightarrow (\sigma_M B_u \onesurv)_{u \in [0,1]} \]
as $t \to \infty$, where $B$ is a standard Brownian motion independent of $\surv$, the event of survival.
\end{remark}

\old{%
\begin{remark} Another way to prove the central limit theorem for $Y_t$ is by applying the CLT for functions of a Markov chain (see{\cite[Theorem 9 Part 6]{Jones}}) to the set of state-offspring pairs. This approach requires us to assume the environment chain is aperiodic. The resulting variance is given by:
	\[\sigma^2_{\text{C}} := \sigma^2 + \sum_{i \in S} \pi_i\mu_i^2 -\mu^2 + 2\sum_{t=1}^\infty \left(\sum_{i,j \in S} (\mu_i-1)(\mu_j-1)(P^t)_{ij}\pi_i - (\mu-1)^2\right).\]
Note this is an infinite sum, in contrast to the finite sum defining $\sigma_M^2$.
It would be interesting to find a direct proof of the equality $\sigma^2_{M} = \sigma^2_{\text{C}}$.
\end{remark}
}%

\section{Matrix Generating Function}\label{sec: generating functions}

\subsection{Extinction matrix}
In Theorem~\ref{t.survival.intro} we obtained qualitative results about the survival of BPME. In this section we introduce a matrix generating function to obtain quantitative estimates of the extinction probabilities. 

In the ordinary Galton-Watson branching process with offspring distribution $(p_0,p_1,p_2,\dots)$, the generating function for the offspring distribution is given by:
	\[f(x):=\sum_{k=0}^\infty p_k x^k, \quad |x|\leq 1.\]
Many elementary branching process results can be obtained by analyzing the generating function \cite{Harris, Athreya}. For instance,
\begin{itemize}
\item If $q$ is the extinction probability of the branching process, then $q$ is the smallest fixed point of $f$ in $[0,1]$. 
\item If $\mu \leq 1$, then $q=1$. 
\item If $\mu >1$, then $q \in [0,1)$ and $q$ is the unique fixed point of $f$ in $[0,1)$.
\item For every $t \in [0,1)$, we have $\lim_{n \to \infty} f^n(t)= q$, where $f^n$ refers to the $n$th iterate of $f$.
\end{itemize}
We will prove some analogous results for BPME, namely that the extinction matrix is a fixed point of the generating function, and that iterates of the generating function starting at any matrix (entrywise) between the zero matrix and the extinction matrix converge to the extinction matrix.

Recall that $P$ denotes the environment chain transition matrix and $R$ denotes the reproduction matrix: $R_{jn}$ is the probability of producing $n$ offspring if the environment state is $j$. 
For $n \in \N$ we define the $S \times S$ matrix $P_n$ by
	\[ (P_n)_{ij}:=P_{ij}R_{jn} \] 
the probability that environment $i$ transitions to environment $j$ and $n$ offspring are produced.

A nonnegative matrix $M$ with real entries is called \textbf{stochastic} if all of its row sums are $1$, and \textbf{substochastic} if all of its row sums are $\leq 1$. Note that $P_n$ is substochastic for each $n$. 
We define the \textbf{matrix generating function} 
	\begin{align}
	f(M)&= \sum_{n =0}^\infty P_nM^n \label{eq:gf}
	\end{align}
where we interpret $M^0=I$ (the $S \times S$ identity matrix).
% for any $M \in \R^{S\times S}$.

We make the following observations about $f(M)$.

\begin{lemma}
Let $f$ be the matrix generating function of a BPME. Then 
\begin{itemize}
\item $f(M)$ converges for all substochastic matrices $M$. 
\item If $M$ is substochastic, then $f(M)$ is substochastic.
\item If $M$ is stochastic, then $f(M)$ is stochastic.
%\item If the $i$th row sum of $M$ is $\leq 1-\epsilon$, then the $i$th row sum of $f(M)$ is $\leq 1-\epsilon(1-\prob(1.i \to 0.i))$.
%\item If all row sums of $M$ are $\leq 1-\epsilon$ then the $i$th row sum of $f(M)$ is $\leq 1-\epsilon(1-\prob(1.i \to 0.*))$.
%\item $f$ is increasing in each coordinate.
\end{itemize}
\end{lemma}

\begin{proof}
Note that $\sum_{n=0}^\infty P_n = P$, the transition matrix of the environment chain. 
Writing $\one$ for the all $1$'s vector and $\leq$ for coordinatewise inequality of vectors, a matrix $M$ with nonnegative entries is substochastic if and only if $M \one \leq \one$, and equality holds if and only if $M$ is stochastic.

Let $M$ be substochastic. Since all entries of $M, P_0, P_1, \ldots$ are nonnegative, we have
	\[ f(M)\one = \sum_{n = 0}^\infty P_n M^n \one \leq \sum_{n=0}^\infty P_n \one = P \one = \one \]
and if $M$ is stochastic then equality holds.
\end{proof}

\old{%detailed proof
\begin{proof}
Let $M$ be a substochastic matrix. Then $M^n$ is substochastic for every $n \geq 0$. Then the $(i,j)$ entry of $f(M)$ is:
\begin{align*}
\sum_{n=0}^\infty (P_nM^n)_{ij} &= \sum_{n=0}^\infty \sum_{k \in S} (P_n)_{ik} (M^n)_{kj}\\
& \leq \sum_{n=0}^\infty \sum_{k \in S} (P_n)_{ik}\\
&=\sum_{n=0}^\infty \sum_{k \in S} P_{ik} R_{k,n}\\
&= \sum_{k \in S} P_{ik} \underbrace{\sum_{n=0}^\infty R_{k,n}}_{1}\\
%&=\sum_{k \in S} P_{ik}\\
&=1.
\end{align*}
Here the exchange in the order of summation is justified since all terms are nonnegative.
Thus $f(M)$ converges for every substochastic matrix $M$. Next we compute the $i$th row sum of $f(M)$:
\begin{align*}
\sum_{j \in S} (f(M))_{ij} &= \sum_{j \in S}  \sum_{n=0}^\infty (P_nM^n)_{ij}  \\
&= \sum_{n=0}^\infty  \sum_{j \in S} \sum_{k \in S} (P_n)_{ik} (M^n)_{kj}\\
&= \sum_{n=0}^\infty \sum_{k \in S} (P_n)_{ik} \sum_{j \in S} (M^n)_{kj}.
\end{align*}
Again the exchange in the order of summation is justified since all terms are nonnegative. If $M$ is substochastic, then the $i$th row sum of $f(M)$ is:
\begin{align*}
\sum_{n=0}^\infty \sum_{k \in S} (P_n)_{ik} \sum_{j \in S} (M^n)_{kj} & \leq \sum_{n=0}^\infty \sum_{k \in S} (P_n)_{ik} \\
&= \sum_{n = 0}^\infty \sum_{k \in S} P_{ik} R_{k,n}\\
&=\sum_{k \in S} P_{ik} \underbrace{\sum_{n =0}^\infty R_{k,n} }_{1}\\
%&=\sum_{k \in S} P_{ik}\\
&\leq 1
\end{align*}
with equality in the first and last line in the case that $M$ is stochastic. 
\end{proof}
}%

%As a consequence of the above,
%\begin{itemize}
%\item 
%By compactness, a subsequence of $\{f^n(M)\}_{n \geq 1}$ converges (to a substochastic matrix) for every $M$.
%\item If we require $\forall 1 \leq i \leq n,\prob(1.i \to 0.i) \neq 1$ (i.e., no state leads to immediate extinction almost surely), then if the $i$th row sum of $M$ is less than $1$ we also have the $i$th row sum of $f(M)$ is less than $1$ (i.e., $M$ substochastic implies $f(M)$ substochastic).
%\end{itemize}

For integers $x$ and $y$, and environments $i$ and $j$, denote by $\{x.i \to y.j\}$ the event that the total state transitions from $x.i$ to $y.j$ in one time step; that is, state $i$ transitions to state $j$ and $y-x+1$ offspring are produced. 
This event has probability
	\[ 		
\Pr{x.i \to y.j} = \begin{cases}P_{ij} R_{j,y-x+1} & \text{ if } y \geq x-1 \\ 0 & \text{ else.} \end{cases}
	\]

We now introduce a matrix of extinction probabilities. Recall that $n.i$ denotes the total state with population $n$ and environment $i$. We say that the initial total state $X_0.Q_0=n.i$ \emph{halts in $0.j$} if $T:=\inf\{t ~:~ X_t=0\}$ is finite and satisfies $Q_T=j$.  The \textbf{extinction matrix} is the $S \times S$ matrix $E$ with entries
	\[ E_{ij}= \Pr{1.i \text{ halts in } 0.j}. \]
Note that $E$ is substochastic, since for all $i \in S$
	\[\sum_{j\in S} E_{ij} = \Pr{1.i \text{ halts in } 0.j \text{ for some }j}  \leq 1. \]
Let $E^n$ be the $n$th power of the extinction matrix, and let $f$ be the matrix generating function \eqref{eq:gf}.

\begin{lemma}\label{lemma: fixed point}
$(E^n)_{ij}= \Pr{n.i \text{ halts in } 0.j}$, and $f(E)=E$. 
\end{lemma}

\begin{proof}
We prove the first part by induction on $n$. If it holds for $E^n$, then
	\begin{align*}
	(E^{n+1})_{ij} 
	%&= \sum_{k\in S} (E^n)_{ik} E_{kj}\\
		&= \sum_{k \in S} \Pr{n.i \text{ halts in } 0.k} \Pr{1.k \text{ halts in } 0.j}.
	\end{align*}
The population must reach $1$ before reaching $0$, as it decreases by at most one per time step. 
Now $\Pr{n.i \text{ halts in } 0.k} $ is also the probability that the BPME started at $(n+1).i$ eventually reaches a population of $1$ individual, and the first time it does so it is in environment $k$. Hence, the above sum is equal to $\Pr{(n+1).i \text{ halts in } 0.j}$, completing the induction.

Now observe that
	\begin{align*}
	(P_nE^n)_{ij} 
	%&= \sum_{k \in S} (P_n)_{ik} (E^n)_{kj} \\
		&= \sum_{k \in S} \Pr{1.i \to n.k} \Pr{n.k \text{ halts in } 0.j}.
		%&=\Pr{1.i \text{ transitions to population $n$ and then halts in } 0.j}.
	\end{align*}
Summing over $n$, the $(i,j)$ entry of $\sum_{n=0}^\infty P_nE^n$ equals $\Pr{1.i \text{ halts in }0.j}$, which is $E_{ij}$. Thus $f(E)=E$.
\end{proof}

Now we state the main goal of this section.

\begin{theorem}\label{thm: extinction matrix}
Let $O$ be the $S \times S$ zero matrix. Then $\lim_{n \to \infty} f^n(O)=E$.
\end{theorem}

Write $M \leq N$ if $M_{ij} \leq N_{ij}$ for all $i,j \in S$.  Note that if $M,N$ are substochastic and $M \leq N$ then
\begin{equation} \label{eq:increasing} f(M) \leq f(N). \end{equation} 

\begin{corollary}
If $O \leq M \leq E$, then $\lim_{n \to \infty} f^n(M) = E$.
\end{corollary}

\begin{proof}
By \eqref{eq:increasing}, $O \leq M \leq E$ implies $f^n(O) \leq f^n(M) \leq f^n(E)$ for all $n \geq 1$. Taking the limit, we find
	\[E=\lim_{n\to\infty}f^n(O) \leq \lim_{n\to\infty}f^n(M) \leq \lim_{n\to\infty}f^n(E)=E. \qedhere \]
%Hence $\lim_{n \to \infty} f^n(M)=E$.
\end{proof}

\subsection{Extinction in $m$ generations}

To prepare for the proof of Theorem~\ref{thm: extinction matrix}, we develop some notation describing the multi-step transitions of the BPME.

Denote by $\{x.i \xrightarrow{n} y.j\}$ the event that the total state transitions from $x.i$ to $y.j$ in $n$ time steps; that is, there exist total states $x_1.k_1,\dots,x_{n-1}.k_{n-1}$ such that
	\[x.i \to x_1.k_1\to \cdots \to x_{n-1}.k_{n-1} \to y.j.\]

Next we define events $\trans{}{m}$ and $\trans{n}{m}$ to describe how the population can decrease over longer time periods. 
%For example, the event $\{x.i \trans{}{1} (x-1).j\}$ means that one individual fails to produce any offspring, and the environment transitions from $i$ to $j$ in the process. 
For example, the event $\{x.i \trans{}{2} (x-1).j\}$ means that one individual produces any number of offspring, but all offspring of the first individual fail to produce any offspring, and the environment transitions from $i$ to $j$ in the process. 
%Similarly, $\{x.i \trans{}{3} (x-1).j\}$ can be interpreted as allowing an individual, all of their offspring, and all of \emph{their} offspring to reproduce, but all of the offsprings' offspring fail to produce any offspring.
Likewise, $\{x.i \trans{}{m} (x-1).j\}$ can be interpreted as a single individual's family tree going extinct in at most $m$ generations. 
% (Here we assume that individuals reproduce in depth-first order.) 
Finally, $\trans{n}{m}$ means that each of $n$ individuals' family trees go extinct in at most $m$ generations. 

Formally, these events are defined as follows. We first define $x.i \trans{n}{1} y.j$ if and only if $y=x-n$ and $x.i \xrightarrow{n} y.j$. We define $x.i \trans{0}{m} y.j$ if and only if $x=y$ and $i=j$.
For $m,n \geq 2$ we recursively define
\begin{align*}
\{x.i \trans{1}{m} (x-1).j\} &:= \bigcup_{n \in \N} \bigcup_{k \in S}
 \{ x.i \to (x-1+n).k \trans{n}{m-1} (x-1).j\}\\
\{x.i \trans{n}{m} (x-n).j\} &:= \bigcup_{k \in S} \{ x.i \trans{1}{m} (x-1).k \trans{n-1}{m} (x-n).j\}.
\end{align*}
The union defining $\trans{1}{m}$ includes $n=0$, and the corresponding event is $\{x.i\to (x-1).j\}$.
We write $\trans{}{m}$ to mean $\trans{1}{m}$.  

We make a few observations:
\begin{itemize}
\item $\trans{n}{m}$ results in population decrease of exactly $n$. Moreover, the population at the end is strictly smaller than the population at any previous time. 
\item Since extinction in at most $m$ generations implies extinction in at most $m+1$ generations, 
	\begin{equation} \label{eq:monotonicity} \{1.i \trans{}{m} 0.j \} \subseteq \{1.i \trans{}{(m+1)} 0.j\} \end{equation}
as can be verified from the formal definition by induction on $m$.
\item Extinction in $t$ time steps implies extinction in at most $t$ generations. Conversely, extinction in $m$ generations implies extinction in a finite number of time steps. Hence
	\begin{equation} \label{eq:twounions}
	\bigcup_{m \geq 1} \{1.i \trans{}{m} 0.j\} = \bigcup_{t \geq 1} \{1.i \xrightarrow{t} 0.j \}.
	 \end{equation}
\item For all $x,y \in \Z$, we have $\Pr{x.i \trans{n}{m} (x-n).j} = \Pr{y.i \trans{n}{m} (y-n).j}$.
%\item \red{Maybe also: $\{(n+1).i \xrightarrow{\times(n+1)}_m 0.j\} = \bigcup_{k \in S} \{(n+1).i \xrightarrow{\times n}_m 1.k\} \cap \{1.k \xrightarrow{\times 1}_m 0.j\}$?}
\end{itemize}

Some illustrations of these events are shown in Figure~\ref{fig:processes3}. Only the population size is depicted, not the state of the environment. %Note that the choice of starting population was arbitrary.

\begin{figure}
\centering
\begin{tikzpicture}[scale=0.5]
    \foreach \x in {0,1,2,3,4,5,6,7,8}
        \draw (\x cm,1pt) -- (\x cm,-3pt);
            %node[anchor=north] {$\x$};
    \foreach \y in {0,1,2,3,4,5}
        \draw (1pt,\y cm) -- (-3pt,\y cm);% node[anchor=east] {$\y$};
    %\node[anchor=east] at (-3pt, 3 cm) {$x$};
    \draw[->] (0,0) -- node[below=0.5cm]{$t$} (8.5,0);
    \draw[->] (0,0) -- node[left=0.5cm]{$X_t$} (0,5.5);
    \draw (1,3) -- (2,5) -- (3,4) -- (4,3) -- (5,2);
    \draw (5,2) -- (6,3) -- (7,2) -- (8,1);
    \draw plot[mark=*] (1,3);
    \draw plot[mark=*] (2,5);
    \draw plot[mark=*] (3,4);
    \draw plot[mark=*] (4,3);
    \draw plot[mark=*] (5,2);
    \draw plot[mark=*] (6,3);
    \draw plot[mark=*] (7,2);
    \draw plot[mark=*] (8,1);
    \draw [decorate,decoration={brace,amplitude=4pt},xshift=0pt,yshift=-3pt]
(5,2) -- (1,2);
 \draw [decorate,decoration={brace,amplitude=4pt},xshift=0pt,yshift=-3pt]
(8,1) -- (5,1);
\end{tikzpicture}
\begin{tikzpicture}[scale=0.5]
\foreach \x in {0,1,2,3,4,5,6,7,8,9,10,11,12}
        \draw (\x cm,1pt) -- (\x cm,-3pt);
            %node[anchor=north] {$\x$};
    \foreach \y in {0,1,2,3,4,5,6,7}
        \draw (1pt,\y cm) -- (-3pt,\y cm);% node[anchor=east] {$\y$};
    %\node[anchor=east] at (-3pt, 2 cm) {$x$};
    \draw[->] (0,0) -- node[below=0.5cm]{} (12.5,0);
    \draw[->] (0,0) -- node[left=0.5cm]{} (0,7.5);
    \draw (1,2) -- (2,5) -- (3,7) -- (4,6) -- (5,5) -- (6,4) -- (7,5) -- (8,4) -- (9,3) -- (10,2) -- (11,2) -- (12,1);
    \draw plot[mark=*] (1,2);
    \draw plot[mark=*] (2,5);
    \draw plot[mark=*] (3,7);
    \draw plot[mark=*] (4,6);
    \draw plot[mark=*] (5,5);
    \draw plot[mark=*] (6,4);
    \draw plot[mark=*] (7,5);
    \draw plot[mark=*] (8,4);
    \draw plot[mark=*] (9,3);
    \draw plot[mark=*] (10,2);
    \draw plot[mark=*] (11,2);
    \draw plot[mark=*] (12,1);
    \draw [decorate,decoration={brace,amplitude=4pt},xshift=0pt,yshift=-3pt]
(6,4) -- (2,4);
    \draw [decorate,decoration={brace,amplitude=4pt},xshift=0pt,yshift=-3pt]
(9,3) -- (6,3);
	\draw [decorate,decoration={brace,amplitude=4pt},xshift=0pt,yshift=-3pt]
(10,2) -- (9,2);
\draw [decorate,decoration={brace,amplitude=4pt},xshift=0pt,yshift=-3pt]
(12,1) -- (10,1);
\end{tikzpicture}
\caption{Left: An example of the event $\trans{2}{2}$. Right: An example of the event $\trans{}{3}$. Each $\trans{}{2}$ event is marked with a brace. 
}
\label{fig:processes3}
\end{figure}

The next lemma gives an interpretation of the entries of the $n$th power of the $m$th iterate of $f$ applied to the zero matrix.

\begin{lemma}\label{l.iterates}
For all $n \geq 0$, $m \geq 1$, we have 
	\[ (f^m(O)^n)_{ij} = \Pr{n.i \trans{n}{m} 0.j}. \]
\end{lemma}

\begin{proof}
First if $n=0$, then by our convention $M^0 = I$ (the identity matrix), 
%and $\{x.i \trans{0}{m}x.j\} = \{i=j\}$
we have 
	$(f^m(O)^0)_{ij} = I_{ij} = \Pr{0.i \trans{0}{m}0.j}$ for all $m$, 
as desired. Next, if $m=n=1$, then
	$(f^1(O)^1)_{ij} = (P_0)_{ij} = \Pr{1.i \trans{1}{1} 0.j} $, 
as desired.

Now we proceed by induction on the pair $(m,n)$ in lexicographic order. Supposing the lemma holds for the pairs $(m,1)$ and $(m,n)$, we check that it holds for the pair $(m,n+1)$:
\begin{align*}	f^m(O)^{n+1}_{ij} &= \sum_{k \in S} f^m(O)_{ik} f^m(O)^n_{kj}\\
		&=\sum_{k\in S} \Pr{1.i \trans{ 1}{m} 0.k} \Pr{n.k \trans{n }{m}0.j} \\
		&= \sum_{k\in S} \Pr{(n+1).i \trans{1}{m} n.k} \Pr{n.k \trans{n }{m}0.j}  \\
		&= \Pr{(n+1).i \trans{(n+1)}{m} 0.j}.
\end{align*}
Finally, supposing the lemma holds for all pairs $(m,0), (m,1), \ldots$ we check that it holds for the pair $(m+1,1)$:	
\begin{align*}
	(f^{m+1}(O)^1)_{ij} &= \sum_{n=0}^\infty (P_n f^{m}(O)^n)_{ij}\\
		&=\sum_{n=0}^\infty \sum_{k\in S} (P_n)_{ik}(f^m(O)^n)_{kj}\\
		&=\sum_{n=0}^\infty \sum_{k\in S} \Pr{1.i \to n.k}\Pr{n.k \trans{ n}{m}0.j}\\
		%&=\Pr{ \exists n \geq 0, k \in S \text{ such that } 1.i \trans{ 1}{1}n.k \trans{ n}{m} 0.j}\\
		&=\Pr{1.i \trans{ 1}{(m+1)}0.j}.
	\end{align*}
This completes the induction.
\end{proof}

\begin{proof}[Proof of Theorem~\ref{thm: extinction matrix}]
By definition of the extinction matrix,
	\begin{align*}
	E_{ij} &= \Pr{1.i \text{ halts in } 0.j} \\
	&= \Pr{ \bigcup_{t \geq 1} \{ 1.i \xrightarrow{t} 0.j \} } \\
	&= \Pr{ \bigcup_{m \geq 1} \{ 1.i \trans{}{m} 0.j \} } \\
	&= \lim_{m \to \infty} \Pr{ \{ 1.i \trans{}{m} 0.j \} }.
	\end{align*}
In the second to last line we have used \eqref{eq:twounions}, and in the last line we have used \eqref{eq:monotonicity}. By Lemma~\ref{l.iterates}, the right side equals $\lim_{m \to \infty} f^m(O)_{ij}$.
\end{proof}

\section{Open Questions} \label{sec: BPME further questions}

\subsection{Infinite state space}
We assumed the environment Markov chain has a finite state space. We expect our results to extend to positive recurrent Markov chains (perhaps assuming a tail bound on the offspring distributions and hitting times).
We used exponential tails of hitting times to prove Theorem~\ref{t.CLT}, but one can check in the proof that $2+\delta$ moments suffice.

The null recurrent case is more subtle. Here the difficulty is that the random variable $\Delta_1$ (the net number of offspring produced in an excursion from the starting environment) is no longer integrable: $\mathbb{E} \Delta_1^+ = \mathbb{E} \Delta_1^- = \infty$. 

The transient case can have quite different behavior, as shown by the next two examples.

\begin{example}
Let the environment chain be a simple random walk on $\mathbb{Z}^3$, with offspring distribution
	\[ \xi_t^i = \begin{cases} 0 & \text{with probability } (|i|+2)^{-3} \\
					     1 & \text{with probability } 1-(|i|+2)^{-3} \end{cases} \]
where $|i|$ denotes the Euclidean norm of $i \in \mathbb{Z}^3$. Even though $\mathbb{E} \xi_t^i < 1$ for all $i$, the resulting BPME survives with positive probability.  The basic estimate used to prove survival is
	\begin{equation} \label{eq:escape} P(L_r \geq ar^2) \leq c_0 e^{-c_1 a} \end{equation}
where $L_r = \sum_{t = 0}^\infty \ind\{|Q_t| < r\}$ is the total time spent by the random walk in the ball $\{i \in \mathbb{Z}^3 \,:\, |i|< r \}$. This estimate can be used to show that the number of times $t$ such that $\xi_t^{Q_t}=0$ is almost surely finite.
\old{
%detailed proof that n.i survives with positive probability for sufficiently large n
To see this, for $r>0$ write
	\[ L_r = \sum_{t = 0}^\infty \ind\{|Q_t| < r\} \]
for the total time spent by the environment in ball $B_r := \{i \in \mathbb{Z}^3 \,:\, |i|< r \}$.  Since random walk in $\mathbb{Z}^3$ has positive probability to exit $B_{2r}$ in time $r^2$ and never return to $B_r$ thereafter, we have for all $a>0$
	\[ P(L_r \geq ar^2) \leq c_0 e^{-c_1 a} \]
for constants $c_0, c_1 > 0$.  
%By Borel-Cantelli, it follows that $P(L_r \geq r^{2.5} i.o.)=0$. 
Now writing $\xi_t = \xi_t^{Q_t}$ and $N = \sum_{t=0}^\infty \ind \{\xi_{t} = 0\}$ we have
	\begin{align*} \E N &\leq \sum_{n=0}^\infty \Pr{\xi_t = 0 \text{ for some $t$ such that } \floor{2^{n-1}} \leq |Q_t| < 2^n} \\
		&\leq \sum_{n=0}^\infty \Pr{L_{2^n} \geq 2^{2.5n}} + (1-(\floor{2^{n-1}}+2)^{-3})^{2^{2.5n}} \\
		&<\infty.
	\end{align*}
Taking $n>\E N$ it follows that $n.i$ survives with positive probability.
% With a little more work, 1.i also survives with positive probability.
}
\end{example}

\begin{example}
Fix $m>0$. Let the environment chain be a simple random walk on $\mathbb{Z}^3$, with offspring distribution
	\[ \xi_t^i = \begin{cases} 0 & \text{ with probability } 1-(|i|+2)^{-3} \\
					     m(|i|+2)^3 & \text{ with probability } (|i|+2)^{-3}. \end{cases} \]
Then $\mathbb{E} \xi_t^i = m$ for all $i$, but the resulting BPME goes extinct almost surely. The proof uses \eqref{eq:escape} to show that almost surely, only finitely many $\xi_t^{Q_t}$ are nonzero.
\end{example}

\subsection{Uniqueness of the fixed point}

By Theorem~\ref{t.survival.intro}, the extinction matrix $E$ is stochastic if and only if $\mu \leq 1$ and $P \neq P_1$. We list some open questions about the matrix generating function $f(M)=\sum P_n M^n$.
\begin{enumerate}
\item In the case $\mu \leq 1$ and $P \neq P_1$, is the extinction matrix $E$ the unique 
%substochastic 
fixed point of $f$?
\item In the case $\mu > 1$, is $E$ the unique 
%substochastic 
fixed point of $f$ that is not stochastic?  
\item Does it hold for every substochastic matrix $M$ that is not stochastic, that $\lim_{n \to \infty} f^n(M) = E$?
\end{enumerate}

\subsection{Stochastic fixed point?}
Iterates of $f$ starting at the identity matrix have a natural interpretation: $f^n(I)_{ij}$ is the probability starting with population $1$ in environment $i$ that the environment is $j$ after $n$ generations. One might expect that as $n \to \infty$, the environment after $n$ generations would converge to the stationary distribution~$\pi$ on the event of survival, but this is wrong!
Let $v$ be the left Perron-Frobenius eigenvector of the extinction matrix, normalized so that its coordinates sum to $1$. Experiments suggest that as $n \to \infty$, the distribution of the environment after $n$ generations converges to $v$ on the event of survival, which motivates the following conjecture.

\begin{conjecture}
	\[ \lim_{n \to \infty} f^n(I)_{ij} = E_{ij} + \left(1-\sum_{k \in S} E_{ik}\right) v_j. \]
\end{conjecture}

\subsection{Multitype BPME}

Consider a network of BPMEs, where the offspring of each BPME are sent to other BPMEs. Formally, let $G=(V,E)$ be a finite directed graph, with the following data associated to each vertex $v \in V$:
	\begin{enumerate}
	\item A finite set $S_v$ (the \textbf{state space} of $v$).
	\item A stochastic matrix $P_v:\R^{S_v} \to \R^{S_v}$ (the \textbf{transition matrix} of $v$).
	\item A stochastic matrix $R_v:\R^{S_v} \to \R^{\prod_{w} \N}$ (the \textbf{reproduction matrix} of $v$).
	\end{enumerate}
Here the product is over out-neighbors $w$ of $v$. When an individual at vertex $v$ reproduces, the state of $v$ updates according to the transition matrix $P_v$, and the individual at $v$ is replaced by a random number of offspring at each out-neighbor $w$ of $v$. The reproduction matrix $R_v$ specifies the distribution of this offspring vector, which depends on the state of vertex $v$. This process continues unless there are no individuals left, in which case the network is said to \textbf{halt}. The \textbf{abelian property} \cite{an1} ensures that the probability of halting does not depend on the order in which individuals reproduce. Moreover, on the event that the network halts, the distribution of the final states of the vertices does not depend on the order in which individuals reproduce.	
	
If the transition matrix $P_v$ is irreducible, then it has a unique stationary distribution $\pi_v$. Let $\mu_{vw}(i)$ denote the mean number of offspring sent to vertex $w$ when an individual at vertex $v$ reproduces in state $i$. Then the long-term average number of offspring sent from $v$ to $w$ when an individual at vertex $v$ reproduces is
	\[M_{vw}:= \sum_{i \in S_v} \pi_v(i)\mu_{vw}(i). \]
Denote by $\mu$ the Perron-Frobenius eigenvalue of the $V \times V$ matrix $M$.

\begin{conjecture} \label{c.halt}
If $\mu<1$, then the network halts almost surely for any initial state and population.

If $\mu=1$ and there are no conserved quantities, then the network halts almost surely for any initial state and population.

If $\mu>1$, then for sufficiently large initial population the network has a positive probability not to halt.
\end{conjecture}

Here a \textbf{conserved quanity} is a collection of real numbers $a_v$ and functions $\varphi_v : S_v \to \R$ for each $v \in V$, such that
	\[ \sum_{v \in V} a_v X_v + \varphi_v(Q_v) \]
is an almost sure constant, where $X_v$ is the number of individuals at vertex $v$, and $Q_v$ is the state of vertex $v$. 

Conjecture~\ref{c.halt} is a common generalization of Theorem~\ref{t.survival.intro} and the main result of \cite{an2}: The former is the case $\#V = 1$, and the latter is the case that all offspring distributions are deterministic.

\section*{Acknowledgment}
We thank Rodrigo Delgado for his careful reading of an early draft.

\bibliographystyle{amsalpha}
\bibliography{branching-in-a-markovian-environment}

\providecommand{\bysame}{\leavevmode\hbox to3em{\hrulefill}\thinspace}
\providecommand{\MR}{\relax\ifhmode\unskip\space\fi MR }
\providecommand{\MRhref}[2]{%
  \href{http://www.ams.org/mathscinet-getitem?mr=#1}{#2}
}
\providecommand{\href}[2]{#2}
\begin{thebibliography}{KLPP97}

\bibitem[AF02]{AF}
David Aldous and James~Allen Fill, \emph{Reversible {M}arkov chains and random
  walks on graphs}, 2002, Unfinished monograph, recompiled 2014, available at
  \url{http://www.stat.berkeley.edu/$\sim$aldous/RWG/book.html}.

\bibitem[AK71]{AK}
Krishna~B. Athreya and Samuel Karlin, \emph{{On Branching Processes with Random
  Environments: I: Extinction Probabilities}}, The Annals of Mathematical
  Statistics \textbf{42} (1971), no.~5, 1499 -- 1520.

\bibitem[AN04]{Athreya}
K.~B. Athreya and P.~E. Ney, \emph{Branching processes}, Dover Publications,
  Inc., Mineola, NY, 2004, Reprint of the 1972 original [Springer, New York;
  MR0373040]. \MR{2047480}

\bibitem[BL16a]{an1}
Benjamin Bond and Lionel Levine, \emph{Abelian networks {I}. {F}oundations and
  examples}, SIAM Journal on Discrete Mathematics \textbf{30} (2016), no.~2,
  856--874.

\bibitem[BL16b]{an2}
\bysame, \emph{Abelian networks {II}. {H}alting on all inputs}, Selecta
  Mathematica \textbf{22} (2016), no.~1, 319--340.

\bibitem[CL21]{an4}
Swee~Hong Chan and Lionel Levine, \emph{Abelian networks {IV}. {D}ynamics of
  nonhalting networks}, Memoirs of the American Mathematical Society (2021), to
  appear.

\bibitem[Dur19]{Durrett5}
Rick Durrett, \emph{Probability: theory and examples}, Cambridge Series in
  Statistical and Probabilistic Mathematics, vol.~49, Cambridge University
  Press, 2019, Fifth edition.

\bibitem[Har02]{Harris}
Theodore~E. Harris, \emph{The theory of branching processes}, Dover
  Publications, 2002, Corrected reprint of the 1963 original.

\bibitem[Jac10]{age}
Christine Jacob, \emph{A general multitype branching process with age, memory
  and population dependence}, 5th International Workshop on Applied
  Probability, 2010.

\bibitem[KLPP97]{conceptual}
Thomas Kurtz, Russell Lyons, Robin Pemantle, and Yuval Peres, \emph{A
  conceptual proof of the {K}esten-{S}tigum theorem for multi-type branching
  processes}, Classical and modern branching processes, Springer, 1997,
  pp.~181--185.

\bibitem[KS66]{KS}
Harry Kesten and Bernt~P. Stigum, \emph{A limit theorem for multidimensional
  {G}alton-{W}atson processes}, The Annals of Mathematical Statistics
  \textbf{37} (1966), no.~5, 1211--1223.

\bibitem[LP17]{LPW}
David~A. Levin and Yuval Peres, \emph{Markov chains and mixing times}, vol.
  107, American Mathematical Society, 2017.

\bibitem[Luc91]{batch}
David~M Lucantoni, \emph{New results on the single server queue with a batch
  {M}arkovian arrival process}, Communications in Statistics. Stochastic Models
  \textbf{7} (1991), no.~1, 1--46.

\end{thebibliography}
\end{document}